\documentclass[letterpaper, 10 pt, conference]{ieeeconf}
\IEEEoverridecommandlockouts                              % This command is only needed if 
\title{\LARGE \bf
Solving Strongly Convex and Smooth Stackelberg Games Without Modeling the Follower
}

\author{Yansong Li and Shuo Han% <-this % stops a space
\thanks{The authors are with the Department of Electrical and Computer Engineering,
        University of Illinois Chicago, Chicago, IL 60607. Email:
        {\tt\scriptsize  \{yli340,hanshuo\}@uic.edu}. This research was supported by the Army Research Office under grant number W911NF-22-1-0034.}%
}

\usepackage{amsmath,amssymb}
\usepackage{amsthm}
\usepackage{bbm}
\usepackage[nocompress]{cite}
\allowdisplaybreaks[3]

\usepackage{graphicx}
\DeclareMathOperator*{\newargmin}{arg\,min}

\theoremstyle{plain}
\newtheorem{proposition}{Proposition}
\newtheorem{theorem}{Theorem}
\newtheorem{corollary}{Corollary}

\theoremstyle{definition}
\newtheorem{definition}{Definition}
\newtheorem{example}{Example}
\newtheorem{assumption}{Assumption}
\usepackage[colorlinks=false,hypertexnames=false]{hyperref} % hypertexnames=false is needed for autonum
\usepackage{mathtools}
\usepackage{autonum}

\usepackage{xcolor}
\definecolor{darkgreen}{rgb}{0,0.5,0}

\begin{document}

\maketitle
\thispagestyle{empty}
\pagestyle{empty}

% \title{Stackelberg Game for Human-Robot Interaction}

% \author{Yansong Li\thanks{Email: \texttt{yli340@uic.edu}}}
% \author[1]{Shuo Han}
% \affil[1]{University of Illinois at Chicago}

% \maketitle
\begin{abstract}
    Stackelberg games have been widely used to model interactive decision-making problems in a variety of domains such as energy systems, transportation, cybersecurity, and human-robot interaction. However, existing algorithms for solving Stackelberg games often require knowledge of the follower's cost function or learning dynamics and may also require the follower to provide an exact best response, which can be difficult to obtain in practice. To circumvent this difficulty, we develop an algorithm that does not require knowledge of the follower's cost function or an exact best response, making it more applicable to real-world scenarios. Specifically, our algorithm only requires the follower to provide an approximately optimal action in response to the leader's action. The inexact best response is used in computing an approximate gradient of the leader's objective function, with which zeroth-order bilevel optimization can be applied to obtain an optimal action for the leader. Our algorithm is proved to converge at a linear rate to a neighborhood of the optimal point when the leader's cost function under the follower's best response is strongly convex and smooth.
% Many tasks in human-robot interaction (HRI) can be modeled as a two-player Stackelberg game, in which one player, known as the leader, tries to minimize her cost assuming the other player, known as the follower, takes his best action.     
% Existing algorithms for solving Stackelberg games often require knowing the cost function or the learning dynamics of the follower.
% Nevertheless, these requirements become too restrictive for tasks in HRI, where the follower is a human, whose cost function and learning dynamics are difficult to obtain. 
% In an attempt to mitigate the restrictive requirements imposed by current algorithms, we develop a gradient-based algorithm that models the follower as an inexact best-response oracle. Namely, our algorithm only requires the follower to provide an approximately optimal action in response to any action played by the leader. Our algorithm is shown to converge to a neighborhood of the optimal point at a linear rate both in theory and in practice when the leader's cost function under the follower's best response is strongly convex and smooth.
\end{abstract}

\section{Introduction}

Stackelberg games have recently been used as a mathematical formulation for a number of tasks in energy systems~\cite{wang_game-theoretic_2014,li_hierarchical_2022,wang_jamming_2018,borndorfer_stackelberg_2012}, transportation~\cite{yang_stackelberg_2007,zhang_competitive_2011}, cybersecurity~\cite{pawlick_game-theoretic_2019}, and human-robot interaction~\cite{nikolaidis_game-theoretic_2017,nikolaidis_mathematical_2017-2,tian_safety_2022}. In a two-player Stackelberg game, one player, which is commonly referred to as the leader, is trying to minimize her cost assuming the other player, which is commonly referred to as the follower, is taking his optimal action after observing the leader's action. In other words, the follower chooses a best response to the leader's action. The Stackelberg game can also be viewed as a special case of bilevel optimization~\cite{colson_overview_2007}. 

% Stackelberg games have recently been used as a mathematical formulation for a number of tasks in human-robot interaction (HRI)~\cite{nikolaidis_game-theoretic_2017,nikolaidis_mathematical_2017-2,tian_safety_2022}. In a two-player Stackelberg game, one player, which is commonly referred as the leader, is trying to minimize her cost assuming the other player, which is commonly referred as the follower, is taking his optimal action after observing the leader's action. In other words, the follower chooses a best response to the leader's action. The Stackelberg game can also be viewed as a special case of bilevel optimization~\cite{colson_overview_2007}. In HRI, the human is usually regarded as the follower while the robot as the leader. 

Many learning algorithms for Stackelberg games or, more generally, bilevel optimization, require knowing the follower's cost function. For example, Fiez et al.~\cite{fiez_implicit_2020} developed a learning algorithm for two-player Stackelberg games that locally converges to an optimal solution. In their algorithm, both the follower and the leader use the gradient descent algorithm, in which the leader uses an inexact gradient that depends on the first-order and second-order information on the follower's cost function. Chen et al.~\cite{chen_single-timescale_2022-1} made the same assumption that the follower's cost function is known and speeded up the algorithm by Fiez et al.\ by adding a predictive term to the follower's learning dynamics.

When the follower's cost function is unknown, existing algorithms require knowledge of the follower's learning dynamics. For example, Conn et al.~\cite{conn_bilevel_2012-1} designed a learning algorithm for bilevel optimization that does not require first-order information on the follower's cost function. However, they also designed the follower's learning dynamics so that the follower is capable of providing an inexact best response. The inexact best response from the designed learning dynamics can be used by the leader to update her action. There are many ways to design the follower's dynamics that yield an inexact best response; see~\cite{mersha_direct_2011-1,zhang_bilevel_2014-2}. There are also algorithms for bilevel optimization that do not rely on designing the follower's dynamics. However, those algorithms assume specific forms of either the follower's optimization problem (e.g., a linear program~\cite{zare_bilevel_2020}) or the leader's problem (e.g., a least-squares problem~\cite{ehrhardt_inexact_2021}).    

Assumptions of a known cost function or known learning dynamics are often made to the follower in previous works on Stackelberg games. However, these assumptions do not hold in many cases. For instance, in human-robot interaction, the follower is a human agent, whose cost function may not be known, and whose learning dynamics may not follow what is given by existing learning algorithms. Furthermore, humans do not always provide a best response, such a phenomenon is known as bounded rationality. Therefore, it is important to consider and relax these assumptions to effectively design autonomous systems that can collaborate with human agents. Existing work on Stackelberg games with a human follower typically assumes that the follower's cost function is given or attempts to learn the cost function by, e.g., inverse reinforcement learning~\cite{hadfield-menell_cooperative_2016,sadigh_planning_2016}. Some results do not assume knowing the follower's cost function but assume that the follower always provides an exact best response~\cite{nikolaidis_human-robot_2017-3,nikolaidis_mathematical_2017-2}. 

% When the follower is a human agent, these assumptions often become invalid: The human's cost function may be unknown, or the human's learning dynamics may not follow what is given by existing learning algorithms. Furthermore, the human does not always give a best response, a phenomenon known in the literature as bounded rationality. While the issue of unknown follower's (i.e., the human's) cost function has also been recognized in HRI, current solutions are far from satisfactory. Existing work in HRI typically assumes that the follower's cost function is given or attempts to learn the cost function by, e.g., inverse reinforcement learning~\cite{hadfield-menell_cooperative_2016,sadigh_planning_2016}. Some results do not assume knowing the follower's cost function but assume that the follower always provides an exact best response~\cite{nikolaidis_human-robot_2017-3,nikolaidis_mathematical_2017-2}. 

\paragraph*{Contribution}
In this paper, we make an initial attempt to relax the current assumptions on the follower when solving two-player Stackelberg games. In particular, our problem setup makes the following assumptions:
\begin{enumerate}
    \item The follower's cost function is unknown and does not follow a specific form.
    \item The follower's learning dynamics are unknown and cannot be designed by us.
    \item The follower only gives an inexact best response that is $ \varepsilon $ close to the exact best response (to be defined formally in Section~\ref{sec:ibr}).
\end{enumerate}
Our end result is a gradient-based learning algorithm that is guaranteed to converge to a neighborhood of the optimal point at a linear rate, where the size of the neighborhood depends on the inexactness $ \varepsilon $ of the best response oracle. Our algorithm is tested numerically for quadratic cost functions, and the numerical simulation results are consistent with our theoretical analysis.

\section{Background: Stackelberg Games}

We consider a two-player Stackelberg game, where $x \in \mathbb{R}^n$
is the leader's action, and $y \in \mathbb{R}^m$ is the follower's action. The goal of the leader
is to minimize her cost $f_1$ assuming that the follower will take the
best response, i.e., the follower's action is optimal with respect to
the follower's cost function given a leader's action. Mathematically, the leader
attempts to solve the following bilevel optimization problem:
\begin{equation}\label{eq:sg}
  \underset{x \in \mathbb{R}^n}{\min.} \quad f (x) \triangleq f_1 (x, r (x)),
\end{equation}
where 
$
  r (x) = \newargmin_{y \in \mathbb{R}^m} f_2 (x, y)
$
is the best response of the follower. In this paper, we assume that the minimizer of $ f_2 (x, \cdot)$ is unique for each $x$ to ensure that the best response function $r$ is 
well-defined. We denote $ D_jf_i(\cdot,\cdot) $ as the derivative mapping of $ f_i $ on the $j$-th argument, where $ i=1,2 $ and $ j=1,2 $. The derivative mapping of the best response $ r $ is denoted as $ Dr $. The optimal solution of~\eqref{eq:sg} is denoted as $ x^{\star} $.

The Stackelberg game formulation~\eqref{eq:sg} can be used for modeling a variety of tasks. For example, in human-robot collaboration, the robot is modeled as the leader and the human as the follower. With a human follower, we
typically have access to the exact zeroth-order and first-order information of the leader's cost
function $f_1$, but the follower's cost function $f_2$ and
best response function $r$ are typically unknown. We adopt these settings on $ f_1 $, $ f_2 $, and $ r $  together with the following assumptions:
\begin{assumption}
    $D_1 f_1 (x, \cdot)$ is $L_{f_1}^x$-Lipschitz, and $D_2 f_1 (x, \cdot)$
is $L_{f_1}^y$-Lipschitz for any $x$.\label{ass:lip}
\end{assumption}
\begin{assumption}
    The best response function $r$ is differentiable, and $\| Dr (x)
    \| \leq R_1$ for any $x$.\label{ass:R1}
\end{assumption}
\begin{assumption}
    The best response function $r$ is $R_2$-smooth, i.e., $ D r $ is $ R_2 $-Lipschitz continuous. \label{ass:R2}
\end{assumption}
% \begin{assumption}
%     There exists a constant $\kappa > 0$ such that $\| D_2 f_1 (x, r (x)) \| \leq \kappa
%     \| \nabla f (x) \|$ for any $x$.\label{ass:kappa}
% \end{assumption}

Throughout this paper, we use $ \|\cdot\| $ to denote the $ L_2 $-norm for vectors and the induced $ L_2 $-norm for matrices. 

Assumption~\ref{ass:lip} is a common
assumption in bilevel optimization and Stackelberg games; see~\cite{colson_overview_2007} and~\cite{fiez_implicit_2020}. Assumptions~\ref{ass:R1} and~\ref{ass:R2} imply that the follower does not change his
action rapidly when the leader's action changes.

% The implications of Assumption~\ref{ass:kappa} will be discussed later in Section~\ref{sec:kappa}. 

\section{Proposed Framework}

In this section, we give a gradient-based algorithm in Section~\ref{sec:ibr} that uses an inexact gradient derived in Section~\ref{sec:algo} to solve~\eqref{eq:sg}. Computing the inexact gradient requires that the follower give an inexact best response to the leader's action, as defined in Section~\ref{sec:ibr}. 

\subsection{Inexact Best Response}\label{sec:ibr}
Our algorithm relies on the simple fact that the inexact best response should be close to the exact best response. A formal definition is given as follows.
\begin{definition}
  A follower's action $ y $ is called an \emph{$ \varepsilon $-inexact best response ($\varepsilon$-IBR)} to a leader's action $ x $ if $\| y- r
  (x ) \| \leq \varepsilon$.
\end{definition}

If the follower uses gradient descent to compute an $\varepsilon$-IBR when the leader's action $ x $ is fixed, i.e.,
$y_{t + 1} = y_t - \beta D_2 f_2 (x, y_t)$, where $ \beta $ is a stepsize, then condition $\| y_t - r (x
) \| < \varepsilon$ can be verified by the quantity $\| y_t - y_{t - 1}
\|$ if $f_2 (x, \cdot)$ is strongly convex for any $y$.

We choose to avoid defining the inexact best response as a function $y
\colon \mathbb{ R }^{n} \rightarrow \mathbb{ R }^{m}$ such that for some constant $\varepsilon > 0$, the function $y$ satisfies $\| y (x) - r (x) \| \leq
\varepsilon$ for any $x$. For tasks in HRI, this would imply that the human's actions follow a function of
the robot's actions. In other words, each robot's action would correspond to a specific and deterministic action of the human, thus imposing a strong assumption on the human's
behavior.

Our gradient-based algorithm is given by
\begin{equation} \label{eq:learning-dynamics}
  x^{k + 1} = x^k - \alpha g_{x^k},
\end{equation}
where $k=1,2,\dots$ is the index of iteration. The stepsize $ \alpha $ is a fixed constant that should be chosen small enough;  the specific choice of $ \alpha $ will be discussed in Section~\ref{sec:main}. The term $ g_{x^k} $ is an inexact gradient at $ x^k $. From the decomposition of the exact gradient mapping
\begin{equation}
\nabla f(x) = D_1f_1(x,r(x)) + D_2f_1(x,r(x))Dr(x),
\end{equation}
an estimation $ g_x $ is given by
\begin{equation}
  g_x = D_1 f_1 (x, y) + D \hat{\psi}_x \label{eq:inexact-grad}, 
\end{equation}
where $ D_1f_1(x,y) $ approximates $ D_1f_1(x,r(x)) $, and $D\hat{\psi}_x$ approximates $ D_2f_1(x,r(x))Dr(x) $. The follower's action $ y $ is an $ \varepsilon $-IBR of the leader's action $ x $. The computation of  $D\hat{\psi}_x$ only requires an inexact best response of the follower. Detail will be given in Section~\ref{sec:algo}.

\subsection{Computing an Inexact Gradient}\label{sec:algo}

Throughout this subsection, we denote the point where we perform gradient estimation as $ x_0 $, i.e., we would like to use $ g_{x_0} $ as an estimate of $ \nabla f(x_0) $. 
First, define
\[ D f_1 (x, y) \triangleq D_1 f_1 (x, y) + D_2 f_1 (x, y) D r (x) . \]
Thus, we can approximate $\nabla f (x_0)$ based on the above formula to obtain 
\begin{align}\label{eq:exact-grad}
  & \nabla f (x_0) = Df_1(x_0,r(x_0)) 	\approx D f_1 (x_0, y_0)\\
  & \qquad = D_1 f_1 (x_0, y_0) + D_2 f_1 (x_0, y_0) D r (x_0) . 
\end{align}
The part $D_1 f_1 (x_0, r (x_0))$ can be approximated by $D_1 f_1 (x_0, y_0)$,
where $y_0$ is an $\varepsilon$-IBR of $x_0$. Similarly, to approximate the
second part, define
$ \psi  (x) \triangleq D_2 f_1 (x_0, y_0) r (x) . 
$ Note that the definition of $\psi  (x)$ depends on $(x_0, y_0)$. Since
$ D_2 f_1 (x_0, y_0) D r (x_0) = \nabla \psi  (x_0), $
we can approximate $D_2 f_1 (x_0, r (x_0)) D r (x_0)$ by a finite-difference
approximation of $\psi $, which requires evaluating $\psi $ at points near $x_0$. However, an exact evaluation of $\psi $ is not possible, since we do not have the exact best response $r$. To circumvent this issue, we approximate $\psi  (x)$
by $\hat{\psi}_x \triangleq D_2 f_1 (x_0, y_0) y$, where $y$ is an
$\varepsilon$-IBR of $x$. Define each $x_i$ as
\begin{equation}\label{eq:xi}
  x_i = x_0 + \delta v_i, \quad i = 1, 2, \ldots, p, 
\end{equation}
where $v_i$'s is a positive basis of $\mathbb{R}^n$. (See~\cite[Section 2.1]{conn_introduction_2009} for the definition of a positive basis.) Our goal is to approximate $D_2 f_1 (x_0, r (x_0)) D r (x_0)$
from $\{ (x_0, y_0), (x_1, y_1), \ldots, (x_p, y_p) \}$, where each $y_i$ is an
$\varepsilon$-IBR of $x_i$.  Our approximator $D \hat{\psi}_{x_0}$
for $D_2 f_1 (x_0, r (x_0)) D r (x_0)$ is given by
\begin{align}
  D \hat{\psi}_{x_0} &= \newargmin_{D \hat{\psi}} \left\|
    \left[\begin{array}{c}
      \hat{\psi}_{x_0}\\
      \hat{\Psi}_{x_0}
    \end{array}\right] - \left[\begin{array}{cc}
      \mathbf{1}_{p + 1} & M
    \end{array}\right] \left[\begin{array}{c}
      \hat{\psi}_{x_0}\\
      D \hat{\psi}
    \end{array}\right] \right\| \\
    & = \newargmin_{D \hat{\psi}} \|
    \hat{\Psi}_{x_0} - M D \hat{\psi} \|,
\end{align}
where
$
 \hat{\Psi}_{x_0} = (\hat{\psi}_{x_1}, \ldots, \hat{\psi}_{x_p})^{\top},
$ and 
\[   \quad M = \left[\begin{array}{c}
  0\\
  \delta V
\end{array}\right], 
\qquad V = [v_1, \ldots, v_p]^{\top}.
\]
\section{Main Results}
In this section, we show that~\eqref{eq:learning-dynamics} converges at a linear rate. The difference between~\eqref{eq:learning-dynamics} and the standard gradient descent is that the gradient~\eqref{eq:inexact-grad} used in~\eqref{eq:learning-dynamics} is inexact. Thus, if the inexact gradient~\eqref{eq:inexact-grad} is close to the true gradient, a convergence result similar to the standard gradient descent is expected. In Section~\ref{sec:inexactness}, we adopt the proof in~\cite[Theorem 2.13]{conn_introduction_2009} to show that the difference between the inexact gradient computed by~\eqref{eq:inexact-grad} and the true gradient is upper bounded. 
The upper bound relies on a function that can be further upper bounded by a function of $\| \nabla f(x) \|$ based on Assumption~\ref{ass:kappa}. The meaning of Assumption~\ref{ass:kappa} will be discussed in detail in Section~\ref{sec:kappa}. 
By using the upper bound constructed in Section~\ref{sec:inexactness}, we establish a similarity between~\eqref{eq:learning-dynamics} and the standard gradient descent and give a proof of convergence for~\eqref{eq:learning-dynamics} in Section~\ref{sec:main}.

\subsection{Upper Bounding the Error of the Inexact Gradient}\label{sec:inexactness}
The following proposition shows that the error of the inexact gradient is upper bounded.
\begin{proposition}
Suppose $f_1$
  and $f_2$ satisfy Assumptions~\ref{ass:lip}, \ref{ass:R1}, and
  \ref{ass:R2}. The inexact gradient $ g_{x} $ defined in \eqref{eq:inexact-grad} satisfies \label{prpo:inexac}
 \begin{equation}\label{eq:bounding_grad_diff}
  \| \nabla f (x ) - g_{x } \| \leq \varphi (x )
 \end{equation}
  for all $x$, where
  \begin{equation}\label{eq:upperboundfor_gx}
    \varphi (x) \triangleq a \varepsilon + b \| D_2 f_1 (x , r (x )) \|
  \end{equation}
  with $a = L_{f_1}^y R_1 + L_{f_1}^x + b L_{f_1}^y$ \ and $ b = \sqrt{p + 1}(\delta^2 R_2 + \varepsilon) \| M^{\dagger} \|/ 2 $.
\end{proposition}

\begin{proof}
  By definition, at any point $x_0$
  \begin{equation}
  \label{eq:diff}
  \begin{aligned}
    &\| \nabla f (x_0) - g_{x_0} \| \\
    &\qquad \leq \| D_1 f_1 (x_0, r (x_0)) - D_1 f_1(x_0, y_0) \| \\ 
    & \qquad \quad+ \| D_2 f_1 (x_0, r (x_0)) D r (x_0) - D \hat{\psi}_{x_0}\|. 
  \end{aligned}
  \end{equation}
The first part of the right-hand side of~\eqref{eq:diff} can be bounded as
  \begin{align}
  \|D_1f_1 (x_0, r (x_0)) - D_1 f_1 (x_0, y_0) \| &\leq L_{f_1}^x \| r   
     (x_0) - y_0 \| \\
     &\leq L_{f_1}^x \varepsilon . \label{eq:1term}
  \end{align}
  Define $\Psi_{x_0} = (\psi (x_1), \ldots, \psi (x_p))^{\top}$, where each
  $x_i$ is defined in~\eqref{eq:xi}. The second part of the right-hand side of
 ~\eqref{eq:diff} is bounded by
 \begin{multline}\label{eq:sec-part}
  \| D_2 f_1 (x_0, r (x_0)) D r (x_0) - D \hat{\psi}_{x_0} \| \\
  \qquad \leq \| D_2 f_1 (x_0, r (x_0)) D r (x_0) - D_2 f_1 (x_0, y_0) D r (x_0) \| \\ 
  \qquad\quad+ \| D_2f_1 (x_0, y_0) D r (x_0) - D \hat{\psi}_{x_0} \| .
 \end{multline}
  The first part of~\eqref{eq:sec-part} can be bounded by 
  \begin{align}
    &\|D_2 f_1 (x_0, r (x_0)) D r (x_0) - D_2 f_1 (x_0, y_0) D r (x_0) \|\\
    &\qquad\leq \| D_2 f_1 (x_0, r (x_0))  - D_2 f_1 (x_0, y_0) \| \| D r (x_0) \|\\
    &\qquad\leq L_{f_1}^y R_1 \varepsilon.\label{eq:2term}
  \end{align}
  The second part of~\eqref{eq:sec-part} can be bounded by
  \begin{align}
    &  \| D_2 f_1 (x_0, y_0) D r (x_0)  - D \hat{\psi}_{x_0} \|\\
    &\qquad = \| \nabla \psi(x_0) - D \hat{\psi}_{x_0} \| \\
    &\qquad \leq \left\| \left[\begin{array}{c}
      \psi (x_0)\\
      \nabla \psi (x_0)
    \end{array}\right] - D \hat{\psi}_{x_0} \right\| \\
    &\qquad= \left\| \left[\begin{array}{c}
      \psi (x_0)\\
      \nabla \psi (x_0)
    \end{array}\right] - M^{\dagger} \hat{\Psi}_{x_0} \right\| \\
    &\qquad \leq \left\| \left[\begin{array}{c}
      \psi (x_0)\\
      \nabla \psi (x_0)
    \end{array}\right] - M^{\dagger} \Psi_{x_0} \right\|  + \| M^{\dagger}\Psi_{x_0} - M^{\dagger} \hat{\Psi}_{x_0} \| . 
  \end{align}
  To bound $\left\| \left[\begin{array}{c}
    \psi (x_0)\\
    \nabla \psi (x_0)
  \end{array}\right] - M^{\dagger} \Psi_{x_0} \right\|$, define
  \begin{equation}
    h \triangleq M \left[\begin{array}{c}
      \psi (x_0)\\
      \nabla \psi (x_0)
    \end{array}\right] - \Psi_{x_0} . \label{eq:residual}
  \end{equation}
  Thus,
  \begin{align}
    | h_i | &= | \psi (x_0) - \psi (x_i) + \langle \nabla \psi (x_0), x_i -
    x_0 \rangle | \\
    &= \left| \int_0^1 \langle \nabla \psi (x_0)\right. \\
    &\qquad-\left. \nabla \psi (x_0 + t (x_i - x_0)), x_i - x_0 \rangle \,dt \right| \\
    & \leq \| x_i - x_0 \| \int_0^1 \| \nabla \psi (x_0) - \nabla \psi (x_0 +
    t (x_i - x_0)) \| \,dt \\
    & \leq \| x_i - x_0 \| \| D_2 f_1 (x_0, y_0) \| \\
    &\qquad \cdot\int_0^1 \| D r (x_0) - D
    r (x_0 + t (x_i - x_0)) \| \,dt \\
    & \leq \| x_i - x_0 \|^2 \| D_2 f_1 (x_0, y_0) \| R_2 \int_0^1 t \,dt
    \\
    & = \frac{\delta^2 R_2 \| D_2 f_1 (x_0, y_0) \|}{2} . 
  \end{align}
From~\eqref{eq:residual}, 
  \begin{align}
    \left\| \left[\begin{array}{c}
      \psi (x_0)\\
      \nabla \psi (x_0)
    \end{array}\right] - M^{\dagger} \Psi_{x_0} \right\| &=\| M^{\dagger} h\|. 
  \end{align}
  Also,
  \begin{align}
    \| M^{\dagger} h\|& \leq  \| M^{\dagger} \| \| h \| \\
    &\leq  \sqrt{\dim (h)} \frac{\delta^2 R_2}{2} \| M^{\dagger} \| \| D_2 f_1
    (x_0, y_0) \| \\
    & =  \frac{\sqrt{p + 1}}{2} \delta^2 R_2 \| M^{\dagger} \| (\| D_2 f_1
    (x_0, r (x_0)) \| + L_{f_1}^y \varepsilon) .   \label{eq:3term}
  \end{align}
  Finally, we can bound $\| M^{\dagger} \Psi_{x_0} - M^{\dagger}
  \hat{\Psi}_{x_0} \|$ by
  \begin{align}
    &\| M^{\dagger} \Psi_{x_0}  - M^{\dagger} \hat{\Psi}_{x_0} \|  \\
    & \qquad \leq  \|M^{\dagger} \| \| \Psi_{x_0} - \hat{\Psi}_{x_0} \| \\
    & \qquad\leq  \frac{\sqrt{p + 1}}{2} \| M^{\dagger} \| \| D_2 f_1 (x_0, y_0) \|\varepsilon \\
    & \qquad\leq  \frac{\sqrt{p + 1}}{2} \varepsilon \| M^{\dagger} \| (\| D_2 f_1
    (x_0, r (x_0)) \| + L_{f_1}^y \varepsilon) .   \label{eq:4term}
  \end{align}
Combine~\eqref{eq:1term},~\eqref{eq:2term},~\eqref{eq:3term}, and~\eqref{eq:4term} to obtain
\begin{align}
  &\| \nabla f (x_0) - g_{x_0} \| \\
  &\qquad \leq  L_{f_1}^x \varepsilon + L_{f_1}^y R_1 \varepsilon \\
  &\qquad\quad+ \frac{\sqrt{p + 1}}{2} \delta^2 R_2 \| M^{\dagger} \|(\| D_2 f_1 (x_0, r (x_0)) \| + L_{f_1}^y \varepsilon)\\
  &\qquad\quad+  \frac{\sqrt{p +1}}{2} \varepsilon \| M^{\dagger} \| (\| D_2 f_1 (x_0, r (x_0)) \| +L_{f_1}^y \varepsilon) \\
  &\qquad=  L_{f_1}^x \varepsilon + L_{f_1}^y R_1 \varepsilon \\
  &\qquad\quad + \frac{\sqrt{p +1}}{2} \| M^{\dagger} \| (\delta^2 R (\| D_2 f_1 (x_0, r (x_0)) \| +L_{f_1}^y \varepsilon) \\
  &\qquad\quad+ \varepsilon (\| D_2 f_1 (x_0, r (x_0)) \| +L_{f_1}^y \varepsilon)) \\
  &\qquad=  L_{f_1}^x \varepsilon + L_{f_1}^y R_1 \varepsilon + \frac{\sqrt{p +1}}{2} (\delta^2 R_2 + \varepsilon) \| M^{\dagger} \| \\
  &\qquad \qquad \cdot(\varepsilon L_{f_1}^y + \| D_2 f_1 (x_0, r (x_0)) \|) . 
\end{align}
  Since $x_0$ is arbitrary, the proof is finished.
\end{proof}

The value of $\|M^{\dagger} \|$ in the upper bound given in Proposition~\ref{prpo:inexac} can be calculated once a positive basis is chosen. The following corollary gives an expression of $\|M^{\dagger} \|$ under the positive basis that consists of the standard basis and the negative standard basis.
 
\begin{corollary} \label{cor:standard_basis}
Choose each $v_i$ as the positive basis defined as
  \begin{equation}
    V = [\begin{array}{cccc}
      v_1 & v_2 & \ldots & v_{2 n}
    \end{array}]^{\top} = \left[\begin{array}{cc}
      I_n & - I_n
    \end{array}\right] .
  \end{equation}
In this case,
  \begin{equation}
    M = \left[\begin{array}{c}
      0\\
      \delta I_n\\
      - \delta I_n
    \end{array}\right] . \label{eq:standard-basis}
  \end{equation}
  We have
  $ \varphi (x) = a \varepsilon + b \| D_2 f_1 (x , r (x )) \| $
  with $a = L_{f_1}^y R_1 + L_{f_1}^x + b L_{f_1}^y$ \ and
   $b = \sqrt{4 n + 2} \left( \varepsilon / \delta + \delta R_2 \right) / 4 $. 
\end{corollary}
\begin{proof}
  Denote by $ \sigma_{\max}(\cdot) $ and $ \sigma_{\min}(\cdot) $ the maximum and minimum singular values of a matrix, respectively. Notice $M^T M = 2\delta^2 I$. This implies $\sigma_{\min}(M) = \sqrt{2} \delta$. Use the fact $\sigma_{\max}(M^\dagger) = 1/\sigma_{\min}(M)$ to obtain $ \|M^{\dagger}\|= \sqrt{2}/(2\delta) $.
\end{proof}

\subsection{Bounded Sensitivity}\label{sec:kappa}

The only non-constant term in the upper bound given in~\eqref{eq:upperboundfor_gx} is $\| D_2
f_1 (x_0, r (x_0)) \|$. To upper bound this term, we make the following assumption.

\begin{assumption}[Bounded sensitivity]
  There exists a constant $\kappa > 0$ such that $\| D_2 f_1 (x, r (x)) \| \leq \kappa
  \| \nabla f (x) \|$ for any $x$.\label{ass:kappa}
\end{assumption}

With Assumption~\ref{ass:kappa}, this term can be further bounded by $\kappa \|  \nabla f
(x_0) \|$, after which techniques for analyzing the standard gradient descent algorithm can be applied (see Section~\ref{sec:main}). This subsection will discuss the practical implications of Assumption~\ref{ass:kappa}.

Recall that the exact gradient of~\eqref{eq:sg} is defined as 
\begin{equation}
  \nabla f (x) = D_1 f_1 (x, r (x)) + D_2 f_1 (x, r (x)) D r (x). 
\end{equation}
Under Assumption~\ref{ass:kappa} and Assumption~\ref{ass:R1}, 
\begin{align}
  \| D_2 &f_1 (x,r(x)) \|  \\
  & \leq \kappa \| D_1 f_1 (x, r (x)) + D_2 f_1 (x, r (x)) D r (x)\|\\
  & \leq \kappa ( \| D_1 f_1 (x,r(x)) \| 
  + \| D_2 f_1 (x,r(x)) \| R_1 ).
\end{align}
Rearrange the above equation to obtain
\begin{equation}\label{eq:kappa-implication}
  \frac{1}{\kappa} \leq \frac{\left\|D_1 f_1(x, r(x))\right\|}{\left\|D_2 f_1(x, r(x))\right\|}+R_1
\end{equation}
when $\kappa\neq 0$ and $ \| D_2 f_1 (x,r(x)) \| \neq 0$.
The ratio $\left\|D_1 f_1(x, r(x))\right\| / \left\|D_2 f_1(x, r(x))\right\|$ characterizes the sensitivity of the leader's cost to the follower's action relative to the leader's action when the follower chooses the best response. Based on~\eqref{eq:kappa-implication}, a smaller $ \kappa $ implies that the leader's cost is less sensitive to the follower's action. The existence of $ \kappa $ implies that the sensitivity is bounded, hence the name \emph{bounded sensitivity assumption} for Assumption~\ref{ass:kappa}.

Another intuitive way to understand the bounded sensitivity assumption is to consider fully collaborative Stackelberg games, where $f_{1}=f_{2}$. In this setting, when the follower takes the best response, the leader's cost is only affected by her own action, which implies $ \kappa = 0 $. Formally, since $r(x)$ is the follower's best response to $x$, the optimality condition of the follower's optimization problem is given by $D_{2}f_{2}(x,r(x))=0$. When $f_{1}=f_{2}$, the optimality condition implies $D_{2}f_{1}(x,r(x))=0$,
allowing one to choose $\kappa=0$ to satisfy Assumption~\ref{ass:kappa}.

\subsection{Convergence Analysis for Strongly Convex and Smooth Cost}\label{sec:main}
Under Assumption~\ref{ass:kappa} and Proposition~\ref{prpo:inexac}, we can prove formally that our algorithm converges linearly to a neighborhood of $ x^{\star} $ by using a similar technique for analyzing standard gradient descent with strongly convex and smooth functions~\cite{frasconi_linear_2016}.
\begin{theorem}\label{thm:main}
  Suppose the follower gives an $\varepsilon$-IBR in every step, and the leader uses stepsize
  $\alpha < 1 / L_f$. Also, assume that $f$ is $\mu_f$-strongly convex and
  $L_f$-smooth, and $f_1$ and $f_2$ satisfy Assumptions~\ref{ass:lip},
  \ref{ass:R1}, \ref{ass:R2}, and \ref{ass:kappa}. Let $a = L_{f_1}^y R_1 +
  L_{f_1}^x + b L_{f_1}^y$, and $b = \sqrt{p + 1} (\delta^2 R_2 + \varepsilon)
  \| M^{\dagger} \| / 2$. The algorithm given by~\eqref{eq:learning-dynamics} and~\eqref{eq:inexact-grad} converges to a
  neighborhood of $f^{\star}$ with the following property:
  \[ \limsup_{k \rightarrow \infty} (f (x_{k + 1}) - f^{\star})
     \leq \frac{(2 a b \kappa \varepsilon + a^2 \varepsilon^2)}{2 \mu_f (1 -
     b^2 \kappa^2 - 2 a b \kappa \varepsilon)} \] if 
     \begin{equation}\label{eq:condition}
     b^2 \kappa^2 + 2 a b \kappa \varepsilon - 1 < 0.
     \end{equation}
\end{theorem}
\begin{proof}
  By the Lipschitz continuity of $\nabla f$, we have
\begin{align}\label{eq:inn}
&f (x_{k + 1})  - f (x_k)  \\
&\qquad \leq \langle \nabla f (x_k), x_{k + 1} - x_k\rangle + \frac{L_f}{2} \| x_{k + 1} - x_k \|^2 \\
&\qquad = - \alpha \langle \nabla f (x_k), g (x_k) \rangle + \frac{\alpha^2 L_f}{2} \| g (x_k) \|^2
\end{align}
By~\eqref{eq:bounding_grad_diff},
  \begin{align}
    \varphi (x_k)^2 &\geq \| \nabla f (x_k) - g (x_k) \|^2 \\
    &= \| \nabla f (x_k)
     \|^2 +  \| g (x_k) \|^2 - 2 \langle \nabla f (x_k), g (x_k) \rangle .
  \end{align}
  Rearrange the above inequality to obtain
  \begin{multline}\label{ineq:11}
    - 2 \langle \nabla f (x_k), g (x_k) \rangle \\
    \qquad\leq \varphi (x_k)^2\ - \|\nabla f (x_k) \|^2 -  \| g (x_k) \|^2 . 
  \end{multline}
  Substituting~\eqref{ineq:11} into~\eqref{eq:inn} gives
  \begin{multline}\label{eq:descentwithineactgrad}
    f (x_{k + 1}) - f (x_k) \\
    \leq - \frac{\alpha}{2} \| \nabla f (x_k) \|^2 + \left( \frac{\alpha^2 L_f}{2} - \frac{\alpha}{2} \right) \| g (x_k) \|^2 \\
    + \frac{\alpha \varphi (x_k)^2}{2} . 
  \end{multline}
  From~\eqref{eq:descentwithineactgrad},
  \begin{align}
    &f (x_{k + 1}) - f (x_k)  \\
    &\qquad \leq -\frac{\alpha}{2} \| \nabla f (x_k) \|^2 + \frac{\alpha \varphi (x_k)^2}{2}\\
    &\qquad \leq - \frac{\alpha}{2} \| \nabla f (x_k) \|^2 + \frac{\alpha}{2} (a\varepsilon + b \kappa \| \nabla f (x_k) \|)^2 \\
    &\qquad= - \frac{\alpha}{2} ( (1-b^2\kappa^2) \| \nabla f (x_k) \|^2 - a^2 \varepsilon^2 \\
    &\qquad\quad- 2 a b \kappa \varepsilon \| \nabla f(x_k) \|). 
  \end{align}
  Since $b^2 \kappa^2 + 2 a b \kappa \varepsilon - 1 < 0$, if $\| \nabla f
  (x_k) \| \geq 1$
  \begin{align}
    &f  (x_{k + 1}) - f (x_k)  \\
    &\qquad \leq - \frac{\alpha}{2} ((1-b^2\kappa^2) \| \nabla f (x_k) \|^2\\
    &\qquad \quad- 2 a b \kappa \varepsilon \| \nabla f(x_k) \|^2) + \frac{\alpha a^2 \varepsilon^2}{2} \\
    &\qquad\leq \mu_f \alpha (b^2 \kappa^2 + 2 a b \kappa\varepsilon - 1) (f (x_k) - f^{\star}) \\
    &\qquad \quad+ \frac{\alpha a^2 \varepsilon^2}{2} ,
  \end{align}
  else if $\| \nabla f (x) \| \leq 1$, we have
  \begin{align}
    &f (x_{k + 1}) - f (x_k)  \\
    &\qquad\leq - \frac{\alpha}{2} (\| \nabla f (x_k) \|^2- b^2 k^2 \| \nabla f (x_k) \|^2)\\
    &\qquad\quad + \frac{\alpha}{2} (2 a b \kappa \varepsilon + a^2 \varepsilon^2) \\
    &\qquad\leq  \mu_f \alpha (b^2 k^2 - 1) (f (x_k) - f^{\star}) \\
    &\qquad\quad +\frac{\alpha}{2} (2 a b \kappa \varepsilon + a^2 \varepsilon^2) .
  \end{align}
  Rearrange and subtract $f^{\star}$ on each side of the above formula to obtain
  \begin{multline}
    f (x_{k + 1}) - f^{\star} \leq (1 + \mu_f \alpha (b^2 \kappa^2 + 2 a b \kappa \varepsilon - 1)) \\
    \qquad\quad\cdot(f (x_k) - f^{\star}) + \frac{\alpha}{2} (2 a b \kappa \varepsilon + a^2 \varepsilon^2) .
  \end{multline}
  Note that $\alpha \leq 1 / L_f \leq 1 / \mu_f$. Thus, \ $1 > 1 + \mu_f
  \alpha (b^2 \kappa^2 + 2 a b \kappa \varepsilon - 1) \geq 0$. Take $f (x_k)
  - f^{\star}$ as a Lyapunov function $V (x_k)$ finishes the proof.
\end{proof}
The theorem shows that the algorithm in~\eqref{eq:learning-dynamics} converges to a neighborhood of $ x^{\star} $ linearly. Because $b^2 \kappa^2 + 2 a b \kappa \varepsilon - 1$ is an increasing function of $ \varepsilon $, condition~\eqref{eq:condition} can always be satisfied if $ \varepsilon $ is small enough. We will discuss more on the effect of $ \varepsilon $ in Section~\ref{sec:simu}. 
\section{Numerical Simulation}\label{sec:simu}

In this section, we test our algorithm when both players use a convex quadratic cost. Formally, we define $ f_i \colon \mathbb{ R }^{n}\times \mathbb{ R }^{m} \rightarrow \mathbb{ R } $ for $i=1,2$ such that 

\begin{equation}
  f_i(x, y)=
  \frac{1}{2} \left[\begin{array}{l}
    x \\
    y
    \end{array}\right]^{\top}
    \left[\begin{array}{cc}
    P_i & Q_i \\
    Q_i^{\top} & R_i
    \end{array}\right]
   \left[\begin{array}{l}
    x \\
    y
    \end{array}\right],
\end{equation}
where
\begin{equation}
  \left[\begin{array}{cc}
    P_i & Q_i \\
    Q_i^{\top} & R_i
    \end{array}\right]
\end{equation}
is positive definite. We adopt the same positive basis defined in Corollary~\ref{cor:standard_basis} to compute the inexact gradient $g_x$ defined by~\eqref{eq:inexact-grad}. Note that $$ \nabla  f(x) = \left( P_1 - Q_1  R_2^{-1} Q_2^{\top} + Q_2 R_2^{-1}R_1^{\top}R_2^{-1}Q_2^{\top}\right) x $$ and $ D_2f_1(x,r(x)) = \left( P_1 - Q_1  R_2^{-1} Q_2^{\top} \right) x. $ Thus, $ \kappa $ can be computed by 
\begin{align}
  \kappa &= \frac{\sigma_{\max} \left( P_1 - Q_1  R_2^{-1} Q_2^{\top} + Q_2 R_2^{-1}R_1^{\top}R_2^{-1}Q_2^{\top}\right) }{\sigma_{\min}\left( P_1 - Q_1  R_2^{-1} Q_2^{\top} \right)} ,
\end{align}
where $ \sigma_{\max}(\cdot) $ and $ \sigma_{\min}(\cdot) $ are the maximum and minimum singular values of the corresponding matrix. Also, we restrict our simulation to the case that $ P_1 - Q_1  R_2^{-1} Q_2^{\top} + Q_2 R_2^{-1}R_1^{\top}R_2^{-1}Q_2^{\top} $ is positive definite to match the strong convexity assumption made in Section~\ref{sec:main}.
\subsection{The Effect of Inexact Best Response}

We chose the sampling radius $ \delta $ as $ \delta = 0.1 $, the stepsize $ \alpha $ as $ \alpha = 0.01 $, and the total number of iterations as $ 1000 $. For the problem instance and the parameters we chose, condition~\eqref{eq:condition} is satisfied for $ \varepsilon < 0.447$. Fig.~\ref{fig:variying_epsilon} shows $ \|x - x^{\star}\| $, the distance between our iterate and the optimal point, versus the number of iterations under different choices of $ \varepsilon $. As the plot shows, the algorithm converges to a neighborhood of $x^{\star}$ at a linear rate. Also, for a larger $ \varepsilon $, the steady-state error becomes larger.

For $ \varepsilon = 0.01 $, $ 0.025 $, and $ 0.04 $, the result that the steady-state error increases with $ \varepsilon $ is consistent with the upper bound 
\begin{equation}\label{eq:upper_bound}
  \frac{\left(2 a b \kappa \varepsilon+a^2 \varepsilon^2\right)}{2 \mu_f\left(1-b^2 \kappa^2-2 a b \kappa \varepsilon\right)} 
\end{equation}
given in Theorem~\ref{thm:main}, which is an increasing function of $ \varepsilon $ when condition~\eqref{eq:condition} is satisfied.
However, the numerical experiment suggests that condition~\eqref{eq:condition} given in Theorem~\ref{thm:main} may not be necessary for convergence, as shown in Fig.~\ref{fig:variying_epsilon} when $ \varepsilon = 0.1 $ and $ 0.2 $.

%To understand why the algorithm still converges when $\varepsilon \geq 1$, note that when both $f_1$ and $f_2$ are quadratic,  $ r $ is a linear function, which makes $ f$ a quadratic function. An  $ \varepsilon $-IBR at timestep $ k $ can be viewed as a function $ \text{IBR}(x_k)  \triangleq r(x_k) \pm \varepsilon_k $, where $ \varepsilon_k \leq \varepsilon $. Note that  $ \text{IBR}(x_k) $ is an affine function, which turns $ \hat{f_k}(x_k) \triangleq f_1(x,r(x_k) + \varepsilon_k) $  as a quadratic function adding an affine term, i.e., each $ \hat{f}_k $ has the form 
% \begin{equation}
%   \hat{f}_k(x) = c_1 x^{\top} E x + c_2  x^{\top} F \varepsilon_k + c_3,
% \end{equation}
% where $ c_1 $, $ c_2 $, and $ c_3 $ are constant coefficients. Two matrices $ E $ and $ F $ are independent of time step $ k $. The optimal solution for each $ \hat{f}_k $ is $ x^{\star}_k = -E^{-1} F \varepsilon_k $. Since $ \varepsilon_k \leq \varepsilon $, the distance between $ x^{\star}_k $ and $ x^{\star} $ is bounded by $ \|E^{-1} F\| \varepsilon $ for any $ k $. Thus, at each timestep, our algorithms tends to a neighborhood of $ x^{\star} $ with a radius bounded by $ \|E^{-1} F\| \varepsilon $. \yl{I have changed the epsilonBR function in my simulation to a perturbed one. Now, the plot is more consistent with the proof of Theorem~\ref{thm:main} since it fluctuates around the neighborhood of $x^{\star}$.}

\begin{figure}[thpb]
  \centering
  \includegraphics[width=0.4\textwidth]{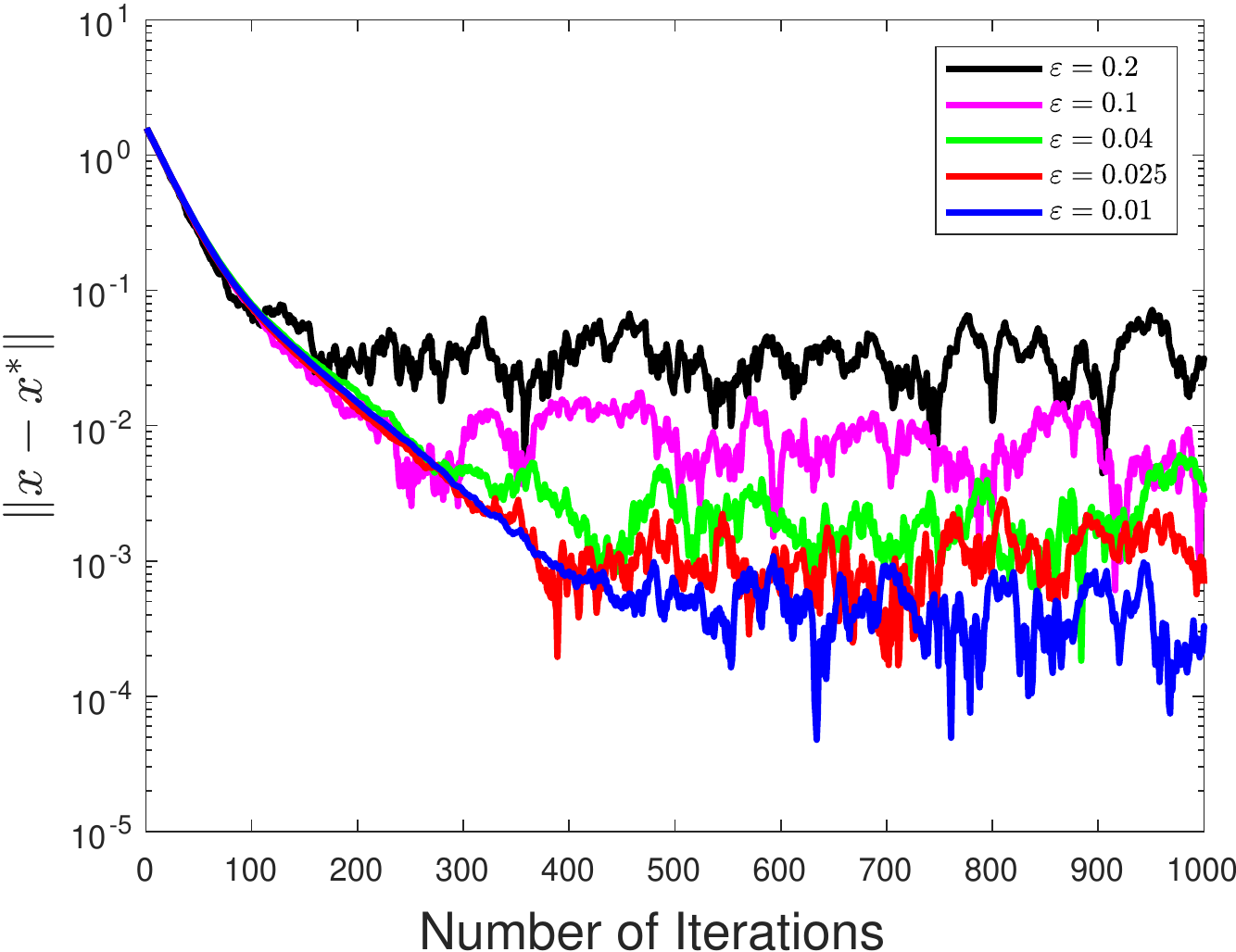}
  \caption{Effect of $ \varepsilon $ on the convergence of~\eqref{eq:learning-dynamics}. The steady-state error increases as $\varepsilon$ increases.}
  \label{fig:variying_epsilon}
\end{figure}

\subsection{Tightness of Error Bound}

In this subsection, we will investigate the tightness of the theoretical error bound~\eqref{eq:upper_bound} given in Theorem~\ref{thm:main} when condition~\eqref{eq:condition} is satisfied. We calculated the gap between the theoretical error bound and the actual error, i.e.,
\begin{equation}\label{eq:tightness}
  \frac{\left(2 a b \kappa \varepsilon+a^2 \varepsilon^2\right)}{2 \mu_f\left(1-b^2 \kappa^2-2 a b \kappa \varepsilon\right)}-(f(x_T) - f^{\star})
\end{equation}
under different choices of $ \varepsilon $, where $ T $ represents the total number of iterations.

It is easy to verify that~\eqref{eq:upper_bound} tends to $ 0 $ as $ \varepsilon $ tends to $ 0 $, which implies that the error bound is tight for $ \varepsilon = 0$, i.e., when the follower gives the exact best response. 
However, our numerical experiment suggests that the theoretical upper bound is loose when $ \varepsilon > 0$, as shown by the relationship between the gap~\eqref{eq:tightness} and $\varepsilon$ given in Fig.~\ref{fig:epsilon}. 
In addition, the gap~\eqref{eq:tightness} increases as $\varepsilon$ increases, which suggests that the theoretical upper bound~\eqref{eq:upper_bound} in Theorem~\ref{thm:main} becomes more conservative as $ \varepsilon $ increases.

\begin{figure}[thpb]
  \centering
  \includegraphics[width=0.4\textwidth]{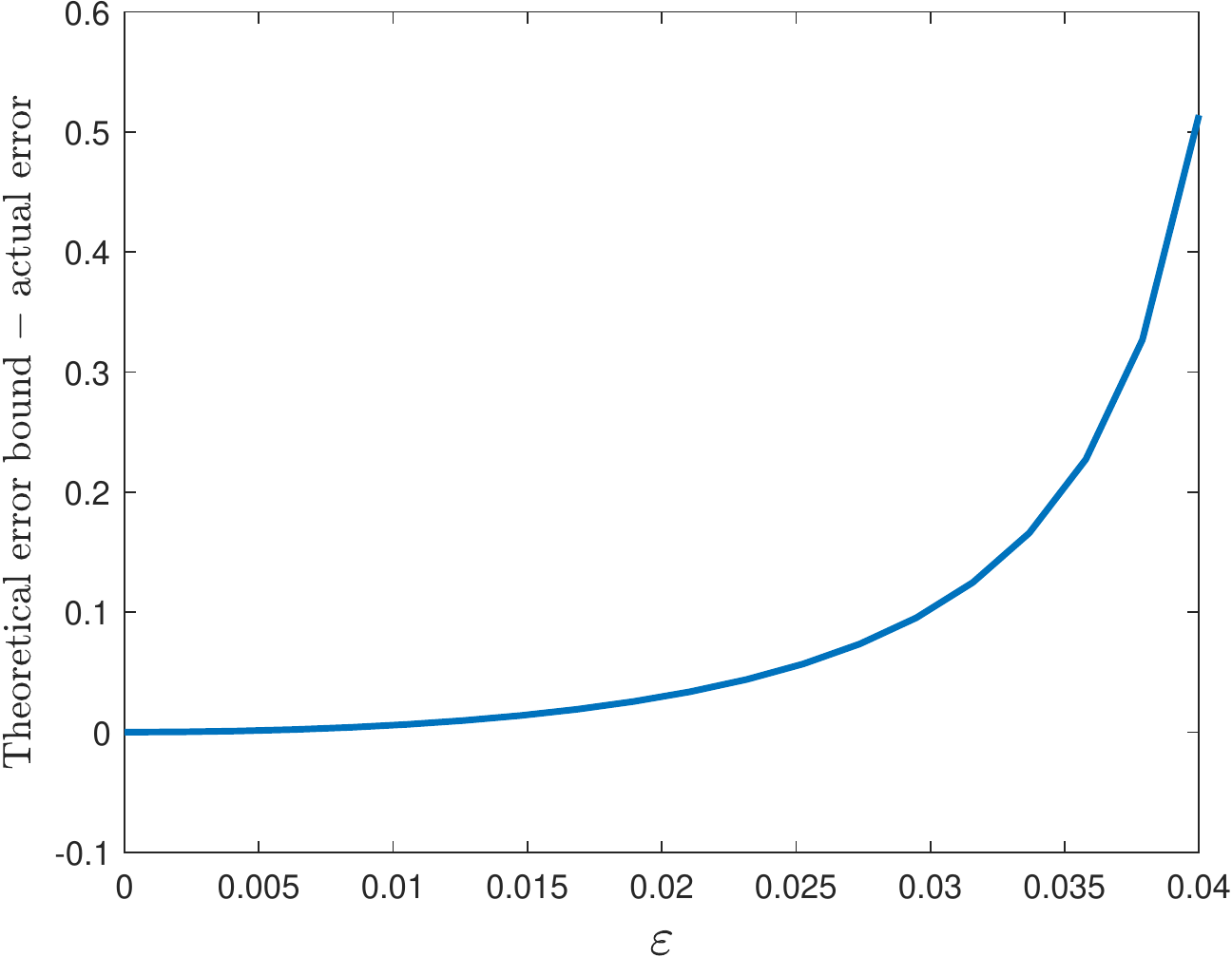}
  \caption{Effect of $ \varepsilon $ on the tightness of the theoretical error bound when $\delta = 0.1$. The theoretical error bound becomes more conservative as $\varepsilon$ increases.}
  \label{fig:epsilon}
\end{figure}

\section{Conclusions}

This paper presents an algorithm for solving Stackelberg games that relaxes several common assumptions in the literature on the follower. Unlike previous work that requires knowledge of the cost function or learning dynamics of the follower, our algorithm only requires the follower to provide an approximate best response under any action played by the leader. This is particularly relevant to interactive decision-making problems in which the follower deviates from rationality and/or has an unknown cost function, for example, when the follower is a human agent. We have shown both theoretically and numerically that our algorithm converges to a neighborhood of the optimal solution at a linear rate.

\bibliographystyle{plain} % We choose the "plain" reference style
\bibliography{StackelbergGame,acc_2023_added} % Entries are in the refs.bib file

\end{document}